\documentclass{article}
\usepackage[textsize=tiny]{todonotes}
 \usepackage{amsthm,amsmath,amssymb, pgf, color, tikz, pgfpict2e}
 \usetikzlibrary{decorations.pathmorphing}
 \usepackage{tocloft}
\usepackage{graphicx} % Required for inserting images
\definecolor{sunglow}{rgb}{1.0, 0.8, 0.2}
\definecolor{aureolin}{rgb}{0.99, 0.93, 0.0}
\definecolor{canaryyellow}{rgb}{1.0, 0.94, 0.0}
\definecolor{fyellow}{rgb}{0.8, 1.0, 0.0}
\definecolor{electricy}{rgb}{1.0, 1.0, 0.0}
\definecolor{tyellow}{rgb}{0.93, 0.9, 0.0}
\theoremstyle{plain}
\newtheorem{thm}{Theorem}[section]
\newtheorem{cor}[thm]{Corollary}
\newtheorem{obs}[thm]{Observation}
\newtheorem{lem}[thm]{Lemma}
\newtheorem{prop}[thm]{Proposition}
\newtheorem{defn}[thm]{Definition}

\newtheorem{exa}[thm]{Example}
\newtheorem{ques}{Question}[section]

\DeclareMathOperator{\prox}{prox}

\title{Discrete-time treatment number}

\author{ N. E.~Clarke \\
\texttt{ \small Acadia Univ.}\\
\texttt{  \small NS Canada}
\and
  K. L.~Collins \\
\texttt{ \small Wesleyan Univ.}\\
\texttt{ \small CT, USA}
\and
  M.E. Messinger \\
\texttt{ \small Mt.Allison Univ.}\\
\texttt{ \small NB Canada}
\and
  A. N. Trenk\\
\texttt{ \small Wellesley Coll.}\\
\texttt{ \small MA, USA}
\and 
  A. Vetta\\
\texttt{ \small McGill Univ.}\\
\texttt{ \small QC Canada}
}

\date{\today}

\begin{document}

\maketitle

\begin{abstract} 
We introduce the discrete-time treatment number of a graph, in which each vertex is in exactly one of three  states at any given time-step: compromised, vulnerable, or treated. Our treatment number is distinct from other graph searching parameters that use only two states, such as the firefighter problem~\cite{FirefighterSurvey} or Bernshteyn and Lee's inspection number \cite{BL21}. Vertices represent individuals and edges exist between individuals with close connections. Each vertex starts out as compromised, and can become compromised again even after treatment. Our objective is to treat the entire population so that at the last time-step, no members are vulnerable or compromised, while minimizing the maximum number of treatments that occur at each time-step. This minimum is the treatment number, and it depends on the choice of a pre-determined length of time $r$ that a vertex can remain in a treated state and length of time $s$ that a vertex can remain in a vulnerable state without being treated again.

We denote the pathwidth of graph $H$ by  $pw(H)$ and  prove that the treatment number of $H$ is bounded above by $\lceil \frac{1+pw(H)}{r+s}\rceil$.  Furthermore, we show that this upper bound equals the best possible lower bound for a \emph{cautious} treatment plan, defined as one in which each vertex, after being treated for the first time, is treated again within every consecutive $r+s$ time-steps until its last treatment.
However, many graphs admit a plan that is not cautious. 
In addition to our results for any values of $r$ and $s$, we focus on the case where treatments protect vertices only for one time-step and vertices remain in a vulnerable state only for one time-step ($r,s=1$). In this case, we provide a useful tool for proving lower bounds, show that the treatment number of an $n\times n$ grid equals $\lceil\frac{1+n}{2}\rceil$, characterize graphs that require only one treatment per time-step, and prove that subdividing an edge of a graph can reduce the treatment number. It is known that there are trees with arbitrarily large pathwidth; surprisingly, we prove that for any tree $T$, there is a subdivision of $T$ that requires at most two treatments per time-step.

\end{abstract}

\medskip

\section{Introduction}\label{sec:intro}

Controlling a contaminating influence on a network is a race between treatment and the possibility of vertices being corrupted again. We model this race with a deterministic graph process called the {\it discrete-time treatment model}.  Initially, every vertex of a graph is considered to be {\it compromised}.  At each time-step, we choose a set of $k$ vertices to {\it treat}.  Informally, in its simplest form, the model uses the following state change rules: 
\begin{itemize}
\item If a \emph{treated} vertex has a compromised neighbor and is not re-treated, it degrades to a {\it vulnerable} state.\vspace{-0.08in}
\item If a vulnerable vertex is not re-treated, it degrades to the compromised state.
\end{itemize}  

We  study the discrete-time {\it treatment number} of a graph, which we define to be the minimum $k$ such that after some finite number of time-steps, no vulnerable or compromised vertices remain.  To visualize the process, we color vertices by state: {\it green} for treated, {\it yellow} for vulnerable, and {\it red} for compromised.  At each time-step we color $k$ vertices green (some of which may already be green) and this corresponds to treating them.   
A treatment protocol with $k=1$ is shown for the graph    $K_{1,3}$   in Figure~\ref{fig:K13}. During time-step $1$, vertex  $a_1$  is colored green; and during time-step $2$, vertex $x$  is colored green.  In time-step $3$, vertex $a_2$ is colored green, but $x$ turns yellow since it has a red neighbor.  During time-step $4$,  vertex $x$  is colored green; and in time-step $5$  vertex $a_3$  is colored green.  By coloring one vertex at each time-step, it is possible to achieve a situation where  no vulnerable or compromised vertices remain in  $K_{1,3}$.

\begin{figure}
\[ \includegraphics[width=0.85\textwidth]{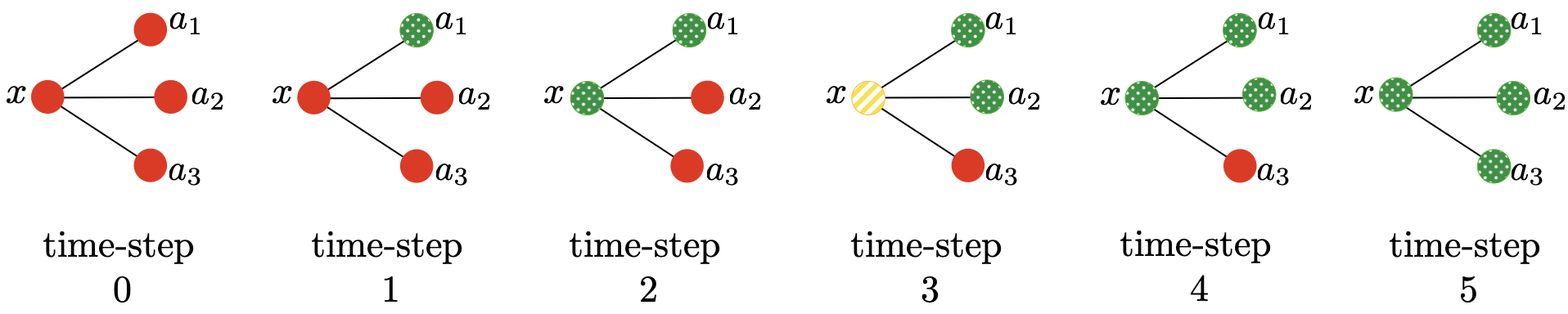} \]

\caption{A visualization of a $(1,1)$-protocol for $K_{1,3} $,  
with green vertices  also indicated by dots and the yellow vertex with diagonal lines. }
\label{fig:K13}
\end{figure}

\smallskip

Our extended time
%generalized 
model, defined in Section~\ref{sec:model}, includes pre-determined lengths of the 
post-treatment protective and vulnerable periods.  This means that when  a vertex is colored green, it remains green for a protective period of $r$ time-steps; and once a vertex turns yellow, it remains yellow in the vulnerable state for $s$ time-steps unless treated.  In the example for $K_{1,3}$ described above, the protective period is one time-step ($r=1)$: a vertex is protected only during the time-step in which it is colored green; the vulnerable period is also one time-step ($s=1$).

Our treatment model   
can be interpreted in many ways, and we give three examples. 
Consider a classroom of children where each child is in one of three states: engaged (green), losing focus (yellow), or distracted (red).  At the start of class, each child is distracted and the teacher's aim is for all children to be engaged.  At each time-step, the teacher directly engages with a subset of $k$ children, turning them green.  The model considers distraction as a corrupting social influence and in the simplest form uses the following state change rules.
If an engaged child has a distracted neighbor and is not directly re-engaged by the teacher, then the child will begin losing focus.
Unless the teacher directly engages with a child who is losing focus, the child will become distracted.
Our goal is to find minimum $k$ for which there is a teacher strategy resulting in all children engaged. 

Alternatively, the vertices of a graph could represent individuals in a population of spies, the edges represent close relationships between individual spies, and the  
three possible states of a vertex are: 
(i) \emph{loyal} (green), (ii) \emph{considering becoming a double agent} (yellow), and (iii) \emph{co-opted by the other side} (red). In this scenario, the treatments are 
meetings with a spymaster to remind individuals  
of their loyalties and perhaps
offer additional rewards. 
After a meeting with a spymaster, an individual is protected from  a change of loyalty for $r$ time intervals, and after being exposed to a double agent, an individual will not turn into a double agent for at least $s$  time intervals.  

A third interpretation of our model is using medical treatments to counteract the spread of an infection in a network. 
The three vertex states in the treatment model correspond to the compartments in Susceptible-Exposed-Infectious-Susceptible (SEIS) epidemic models \cite{BCF2019} where vertices can be (i) \emph{susceptible} (green), (ii) \emph{exposed} (yellow), and (iii) \emph{infectious} (red). 
In this interpretation, individuals are exposed to their close contacts at regular time-steps. Each treatment prevents infection for $r$ time-steps and there are $s$ time-steps before an infected individual becomes contagious. Our discrete-time model differs from the typical SEIS models which assume the number of infections change continuously.

Our discrete-time treatment model is related to, but not the same as, some graph searching problems that also involve dynamically changing vertex states over time.   We briefly discuss three of them: the firefighter problem, node searching, and the inspection game problem.  In the {\it firefighter problem} \cite{FirefighterSurvey}, vertices start out uncolored and  may be colored one of two colors: green (protected) or red (burned).    In  contrast to our  problem, once a vertex is colored in the firefighter problem, it never changes color.  
In node searching \cite{KP}, at each time-step, searchers choose a set of vertices to occupy, just as we choose a set of vertices  to treat.  However, in contrast to our model, in node searching, every {\it edge} of the graph is initially contaminated (i.e., may contain an intruder), and an edge is cleared only when both its endpoints are simultaneously occupied by searchers.  If there is a path from  a cleared edge to a contaminated one that does not contain a searcher, then  the cleared edge is immediately recontaminated.
 For an introduction to node searching, firefighting,  and other pursuit-evasion games, see~\cite{B22}. 

In the inspection game problem \cite{BL21},   searchers inspect a set of $k$ vertices at each time-step to locate an invisible intruder, and the intruder can remain in place or move to a neighboring vertex at each time-step.  The searchers win if there is a time-step in which a searcher occupies the same vertex as the intruder and the inspection number is the minimum $k$ needed to guarantee that the searchers have a winning strategy.  Like our model, the inspection game can be formulated in terms of compromised vertices, that is, during each round, the searchers clear a set of $k$ vertices of the intruder, and then every cleared vertex that has a compromised neighbor becomes compromised again. Our model differs in that there is an intermediate vulnerable state that, to our knowledge, does not appear in other graph searching models. Bernshteyn and Lee prove in \cite{BL21}  that the inspection number of a graph is bounded above by one more than the pathwidth of the graph. Interestingly, we show that the treatment number is bounded above by a smaller function of the pathwidth of a graph, and the size of the function depends on the length of the protective and vulnerable periods. 

\smallskip 

The rest of the paper is organized as follows.
In Section~\ref{sec:model}, we formally introduce the \emph{discrete-time treatment model,} and we define the treatment number of a graph. We provide our fundamental definitions, an example, and preliminary results. 
In Section~\ref{sec:results}, we use pathwidth as a tool to bound the treatment number and study particular types of protocols: minimal, monotone, and cautious. We use the pathwidth to achieve an upper bound for the treatment number and this bound agrees with the best possible lower bound using a cautious protocol. 
Sections~\ref{sec:r=s=1} and~\ref{sec:subdivisions} consider the more restrictive model where the protective period and vulnerable periods each have a length of one time-step.  In Section~4, we characterize graphs with treatment number one, and provide an important tool for proving lower bounds for the treatment number of any graph, which we apply to find the treatment number of the Petersen graph, and any $n\times n$ grid. Bernshteyn and Lee also consider the inspection number of subdivisions of graphs, and this motivates our study of subdivisions in Section~5. We provide an example to show that subdividing one edge in a graph can reduce the treatment number. 
Our main result in this section is that for any tree $T$, there is a subdivision of $T$ whose treatment number is at most 2.  We conclude with a series of questions and directions for future work in Section~\ref{sec:ques}.

Throughout this paper, we assume that every graph is connected and denote the vertex set of graph $H$ by $V(H)$ and its edge set by $E(H)$. 
For general graph theory terms not defined here, please consult a standard reference, such as~\cite{West}. 

 \section{Discrete-time treatment model} 
\label{sec:model}

As noted in the introduction, at each time-step, each vertex has a color: treated vertices are green, vertices that are vulnerable  are yellow, 
and vertices that are compromised are red. 
At time-step 0, all vertices are red. For $t\geq 1$, the set of vertices treated at time-step $t$ is denoted by $A_t$ and is called a {\it treatment set}.  We let $G_t$ be the set of green vertices, $Y_t$ be the set of yellow vertices, and $R_t$ be the set of red vertices at time-step $t$. A newly treated vertex becomes green and stays green for at least $r\ge 1$ time-steps, indicating the strength of the treatment.  
It will turn yellow at the subsequent time-step if it  is not treated and has a compromised neighbor and in this case we say it is \emph{corrupted} or \emph{compromised} by this neighbor.  Otherwise, it remains green at the subsequent time-step.
Once a vertex is yellow, it will become red after $s \ge 1 $ time-steps, unless it is treated again, indicating the vulnerable  period.  
The progression of a vertex from yellow to red occurs  regardless of the status of its neighbors. These transitions are made precise in Definition~\ref{trans} below. 

\subsection{Treatment protocols and color classes} 

In this section we give our formal 
definitions for measuring progress in treatments. 

\begin{defn} \rm For integers $r\ge 1$ and $s \ge 1$, an \emph{$(r,s)$-protocol} for a graph $H$ is a sequence $(A_1, A_2, \ldots, A_N)$ of treatment sets, where $A_i \subseteq V(H)$ for $1 \le i \le N$ and $A_i$ is the set of vertices treated during time-step $t$.  This $(r,s)$-protocol {\it clears} graph $H$ if all vertices are green at time-step $N$; that is, $G_N=V(H)$. 
\end{defn}

While the sequence of treatment sets  in an $(r,s)$-protocol  can be defined without specifying  
$r$ and $s$, we use $r$ and $s$ in the definition because the sequence is relevant only in the context of knowing the values of $r$ and $s$.

\begin{defn} \rm  If $J$ is the $(r,s)$-protocol $(A_1, A_2, \ldots, A_N)$ then its \emph{width}, denoted by  $width(J)$, is the maximum value of $|A_i|$ for $1 \le i \le N$. 
\end{defn} Thus, the  width of an $(r,s)$-protocol $J$ is the maximum number of vertices treated  
during any given time-step.  We sometimes write \emph{protocol} for \emph{$(r,s)$-protocol} when the context is clear.
Every graph  $H$ has an $(r,s)$-protocol of width $|V(H)|$  that clears it, namely $A_1 = V(H)$.  However, our goal is to find the smallest width of an $(r,s)$-protocol that clears a graph.
We next define the $(r,s)$-treatment number of a graph, which is the quantity we seek to minimize.

\begin{defn} \rm 
The \emph{$(r,s)$-treatment number} of a graph $H$, denoted by $\tau_{r,s}(H)$, is the smallest width of an $(r,s)$-protocol that clears $H$.
\end{defn} 
 
We begin with a simple example that shows $\tau_{1,1}(K_{1,m}) = 1$  and is illustrated in Figure~\ref{fig:K13}   
for $m=3$.

\begin{exa} \label{ex-K1m} \rm Let $r=s=1$ and consider $K_{1,m}$, $m\geq 3$, with vertex $x$ of degree $m$ and leaves $a_1, a_2, \ldots, a_m$. 
At time-step $0$, all vertices are red.  Let $A_1 = \{a_1\}$; that is, vertex $a_1$ is treated during time-step $1$.  If nether $a_1$ nor $x$ is  treated during time-step $2$, then $a_1$ will turn yellow during time-step $2$ (because $r=1$).  Let $A_2 = \{x\}$ and $A_3 = \{a_2\}$.  Since $a_3$ is not treated during time-step $3$, $a_3$ corrupts 
$x$ and  $x$ becomes yellow during time-step $3$.  Inductively, let $A_{2j} = \{x\}$, so $x$ is re-treated during time-step $2j$, and let $A_{2j+1}=\{a_{j+1}\}$, $2\leq j\leq m-2$. Although $x$ turns yellow on the odd time-steps, it never becomes red and does not corrupt 
previously treated vertices. In the last two time-steps, $A_{2m-2}=\{x\}$ and  $A_{2m-1} = \{a_m\}$. Then, $a_m$ does not corrupt 
$x$ at time $2m-1$, because it is treated at that time-step. Thus, the $(1,1)$-protocol $(\{a_1\},\{x\},\{a_2\},\{x\},\ldots, \{x\},  \{a_m\})$ clears $K_{1,m}$ because all vertices are green at time-step $2m-1$; that is, $G_{2m-1} = \{x,a_1,a_2,a_3, \ldots, a_m\}$. \end{exa}

We next introduce notation that partitions the set of green vertices at time-step $t$ according to the time-step when these vertices were last treated,  and  the yellow vertices at time-step $t$ according to when they last became re-compromised.

\begin{table}[ht]
\begin{center}
\begin{tabular}{cll} \hline 

\\ 
 
%&&Transitions between classes& \\ \hline
If $v\in A_t$,  &  &then $v\in G_t^r$.  \\

 \vspace{.05in} 
 
If $v\not \in A_t$, & and \\

 \vspace{.05in} 
 
 & $v\in R_{t-1}$,  &then $v\in R_t$.  \\
 
 \vspace{.05in} 
 
%If $v\not \in A_t$, 
  & $v\in Y_{t-1}^1$, & then $v\in R_t$.  \\

   \vspace{.05in} 
   
%If $v\not \in A_t$, 
  & $v\in Y_{t-1}^i$, $2\leq i\leq s$, & then $v\in Y_t^{i-1}$.  \\

   \vspace{.05in} 
   
%If $v\not \in A_t$, 
  & $v\in G_{t-1}^i$, $2\leq i\leq r$,  &then $v\in G_t^{i-1}$.  \\

   \vspace{.05in} 

%If $v\not \in A_t$, 
  & $v\in G_{t-1}^1$  and has a neighbor in $R_t$,  & then $v\in Y_t^s$.  \\

   \vspace{.05in} 
   
%If $v\not \in A_t$, 
  & $v\in G_{t-1}^1$  and has no neighbor in $R_t$,  & then $v\in G_t^1$.  \\ 
  
 \\   \hline 

\end{tabular}

\caption{Transitions between classes of green, yellow, and red vertices.} 

\label{tabchart}

\end{center} 

\end{table}

\begin{defn} \rm \label{trans} 
Let $H$ be a graph with $(r,s)$-protocol $J=(A_1, A_2, \ldots, A_N)$. We  partition the green and yellow vertices at time-step $t$ as follows: 
$G_t = G_t^r \cup G_t^{r-1} \cup \cdots \cup\; G_t^1$ and $Y_t = Y_t^s \cup Y_t^{s-1} \cup \cdots \cup Y_t^1$ where $G_t^r=A_t$ and the remaining sets are defined by the transitions given in Table~\ref{tabchart} and illustrated in Figure~\ref{fig-markov-chains}. 
\end{defn} 

A vertex treated at time-step $t$ is a member of $G_t^r$ and if not treated again, its protection wanes and it transitions to $G_{t+1}^{r-1}$ and then $G_{t+2}^{r-2}$ and so on, as defined in Table~\ref{tabchart}.  A vertex $v \in G_t^1$ will become yellow at time-step $t+1$ (i.e., $v \in Y_{t+1}^s$) if $v \not\in A_{t+1}$ and $v$ has a neighbor in $R_{t+1}$.  A vertex in $Y_{t+1}^s$ transitions to $Y_{t+2}^{s-1}$, and so on, eventually to $Y_{t+s}^1$ and then $R_{t+s+1}$ unless treated. Figure~\ref{fig-markov-chains} illustrates the transitions in Table~\ref{tabchart}. 

\begin{figure}[htb]

\begin{center}
\begin{tikzpicture}[scale=1.2]
\draw[green!40!teal, thick] (1.5,-.5)--(5,-.5)--(5,1)--(1.5,1)--(1.5,-.5);
\draw[green!55!teal] (2,-.5)--(2,1);
\draw[green!55!teal] (2.5,-.5)--(2.5,1);
\draw[green!55!teal] (4.5,-.5)--(4.5,1);
%\draw[blue,thick](3,.3) ellipse (1.5cm and .8cm);
\draw[red,thick](-2,.3) ellipse (1.5cm and .8cm);
\draw[yellow!70!red, thick] (-1,-4)--(2.5,-4)--(2.5,-2)--(-1,-2)--(-1,-4);
\draw[yellow!70!red] (-.5,-4)--(-.5,-2);
\draw[yellow!70!red] (1.5,-4)--(1.5,-2);
\draw[yellow!70!red] (2,-4)--(2,-2);
%\draw[blue,thick](.8,-2.55) ellipse (1.6cm and 1cm);

\node(0)[black] at (-1.95,.4) { compromised };
\node(0)[black] at (-1.95,.1) { (Red) };
\node(0)[black] at (3.5,.4) {treated};
\node(0)[black] at (3.45,0.1) {(Green)};
\node(0)[black] at (3.5,-.25) {$\cdots$};
\node(0)[black] at (.6,-2.6) {vulnerable};
%\node(0)[black] at (.45,-2.45) {vulnerable, but};
%\node(0)[black] at (.5,-2.75) {not compromised};
\node(0)[black] at (.5,-3.05) { (Yellow) };
\node(e)[black] at (.5,-3.4) { $\cdots$ };

%\draw[->,red!45!darkgray, thick] (2.35,-4.1) to [bend left=80] (-.8,-4.1);
\draw[->, blue!45!gray] (2.3,-3.4)--(1.9,-3.4);
\draw[->, blue!45!gray] (1.7,-3.4)--(1.3,-3.4);
\draw[->, blue!45!gray] (-.3,-3.4)--(-.7,-3.4);
\draw[->, blue!45!darkgray, thick] (-3.2,.95) to [bend left=90] (-1.24,1.15);
\node(d) at (-4.2,1.3) {\footnotesize Do not treat};

\draw[->, green!35!darkgray,thick] (4.6,1.2) to [bend right=30] (1.9,1.2);
\draw[->, blue!45!darkgray](4.3,-.25)--(4.7,-.25);
\draw[->, blue!45!darkgray](2.3,-.25)--(2.7,-.25);
\draw[->, blue!45!darkgray](1.75,-.25)--(2.15,-.25);
\node(e) at (3.2,1.85) {\footnotesize Treat };
\node(e) at (6.9,.4) {\footnotesize no compromised neighbor};
\draw[->, blue!45!darkgray,thick] (5.1,0.9) to [bend left=110] (5.1,-.4);

\draw[->, green!35!darkgray,thick](-.4,.3)--(1.355,.3);
\node(a) at (.5,.5) {\footnotesize Treat};
\draw[->, blue!45!darkgray,thick](-.75,-1.9)--(-1.5,-.55);
\node(a) at (-2.15,-1.3) {\footnotesize Do not treat};
%\draw[->,green!35!darkgray, thick](.8,-1.8)--(1.6,-.65);
\draw[->,green!35!darkgray, thick](.7,-1.8) to [bend left=50] (1.4,-.4);
\draw[->, blue!45!darkgray,thick](4.7,-.6) to [bend left] (2.65,-2.25);
\node(b) at (0.1,-.9) {\footnotesize Treat};
\node(c) at (5.63, -1.5) {\footnotesize Do not treat and};
\node(c) at (5.6,-1.78) {\footnotesize compromised neighbor};
\end{tikzpicture}

\caption{A digraph showing the transitions between the states of treated, vulnerable 
and compromised.}
\label{fig-markov-chains}

\end{center} 

 \end{figure}

We record an observation that follows directly from Definition~\ref{trans}.

\begin{obs}
{\rm In an $(r,s)$-protocol for a graph $H$, if a vertex is treated at time $t$ then the earliest it can become red is time-step $t + r + s$.
}
\label{non-comprom-time-obs}
\end{obs}

For the graph $ K_{1,3}$, the $(1,1)$-protocol given in Example~\ref{ex-K1m} is $(\{a_1\}, \{x\}, \{a_2\}, \{x\},\{ a_3\})$. Another $(1,1)$-protocol for $K_{1,3}$ is  $J'=(\{a_1\}, \{x\}, \{a_2\},\{x\},\{a_3\},\{a_1\})$.  Note that it is unnecessary to treat $a_1$ in time-step 6. In a minimal protocol, we do not want to treat vertices frivolously.
  We give a formal definition of \emph{minimal} below.

\begin{defn} \rm  
An $(r,s)$-protocol $(A_1, A_2, \ldots, A_N)$  for graph $H$ is \emph{minimal}  if it satisfies the following for all time-steps $t$ where $1 \leq t \leq N$: if $v \in G^{\ell}_t$ and $v \in A_{t+1}$ then $v$ has a neighbor in $R_{t+\ell} \cup R_{t+\ell+1}\cup \cdots \cup R_{t+r}$.
\end{defn}

We show in Proposition~\ref{min-al} that any $(r,s)$-protocol that clears $H$ can be transformed into a minimal protocol.

\begin{prop} \label{min-al} Let $J$ be an $(r,s)$-protocol $(A_1, A_2, \ldots, A_N)$ that clears $H$. Then there exists a minimal protocol $J'=(A'_1, A'_2, \ldots, A'_N)$ that clears $H$ and for which $A'_i\subseteq A_i$, $1\leq i\leq N$. 
\end{prop} 

\begin{proof} Since $J$ clears $H$, all vertices are in $G_N$. Thus, any vertices that are not green at time-step  $N-1$ must be in $A_{N}$.  Any vertex that is green at time-step $N-1$ remains green at time-step $N$ without needing to be treated, because any of its neighbors that are red at time-step $N-1$ must be in $A_N$.  Hence we can replace $A_N$ by $A'_{N}=A_{N}-G_{N-1}\subseteq Y_{N-1}\cap R_{N-1}$. 

By induction, assume that we have a protocol $(A_1, A_2, \ldots,  A_{k+1}, A'_{k+2}, A'_{k+3}, \ldots, A'_N)$ such that for all time-steps $t$ where $k+1 \leq t \leq N-1$ and any vertex $v \in G^{\ell}_t\cap A'_{t+1}$ with $1\leq \ell \leq r$, then $v$ has a neighbor in $R_{t+\ell} \cup R_{t+\ell+1}\cup \cdots \cup R_{t+r}$ (since $v$ remains green until time-step $t+\ell-1$ without being treated and the protection given to $v$ by being treated at time-step $t+1$ lasts until time-step $t+r$). For the next step in the induction, consider any vertex $u\in G^{\ell}_{k}\cap A_{k+1}$ that does not have a neighbor in $R_{k+\ell} \cup R_{k+\ell+1}\cup \cdots \cup R_{k+r}$. Because the protection $u$ gets from being treated at time-step $k+1$ only lasts until time-step $k+r$, if $u$ has no red neighbors in that time period, then the treatment is unnecessary, and we can remove $u$ from $A_{k+1}$. Let $W$ be the set of such vertices $u$. Then we can replace $A_{k+1}$ by $A'_{k+1}=A_{k+1}-W$, and the protocol $(A_1, A_2, \ldots, A_k, A'_{k+1}, A'_{k+2}, A'_{k+3}, \ldots, A'_N)$ clears $H$. Hence our final protocol $(A'_1, A'_2,  \ldots, A'_N)$ clears $H$ and is minimal. 
\end{proof}

\subsection{Treatments of subgraphs}\label{subsec:subgraphs}

Figure~\ref{fig-caterpillar} shows a graph $H$ with $\tau_{1,1}(H) = 1$ and a subgraph $H'$ of $H$.  In this figure, $(1,1)$-protocols are shown for $H$ and $H'$ by labeling each vertex by the set of time-steps at which it is treated, unlike in Figure~\ref{fig:K13}.  Observe that the protocol $J'$ shown for $H'$ is inherited from the protocol $J$ for $H$ by restricting each treatment set to vertices in $H'$.
The next theorem shows how to generalize this example. 

\begin{figure}[ht]
\[\includegraphics[width=0.835\textwidth]{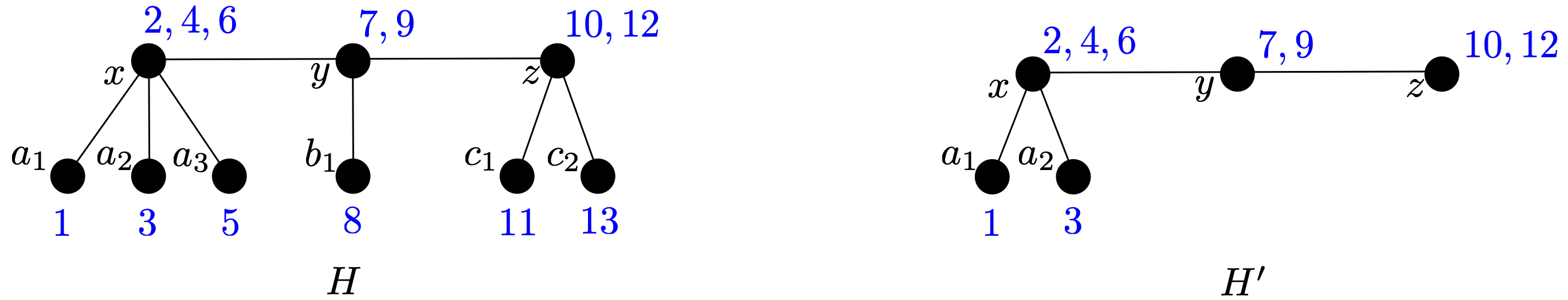}\]
 
\caption{A protocol $J$ that clears $H$ and the restricted protocol $J'$ that clears subgraph $H'$.  A vertex has label $t$ if it is treated at time-step $t$.}
\label{fig-caterpillar}
\end{figure}

\begin{thm} \label{subgraph-lemma} Let $H$ be a graph and $(A_1,A_2,\dots,A_N)$ be an $(r,s)$-protocol that clears $H$. If $H'$ is a subgraph of $H$, and $A'_i=A_i\cap V(H')$, for $1\leq i\leq N$, then graph $H'$ is cleared by $(r,s)$-protocol $(A_1',A_2',\dots,A_N')$.
\label{subgraph-thm}
\end{thm}

\begin{proof} Let $J = (A_1,A_2,\dots,A_N)$ and $J' = (A_1',A_2',\dots,A_N')$. Since $J$ clears $H$, all vertices of $H$ are green at the end of time-step $N$ under protocol $J$.  For a contradiction, suppose that $J'$ does not clear $H'$; hence  there is some vertex in $V(H')$ that is not green at the end of time-step $N$ under protocol $J'$.  Let $t$ be the smallest integer such that there exists $w\in V(H')$ that is green at time-step $t$ under $J$, but not green under $J'$; that is, $w\in G_t(H)$, but $w\not\in G_t(H')$. By the minimality of our choice of $t$, we know $w\in G_{t-1}(H)\cap G_{t-1}(H')$. Since $w\not \in G_t(H')$, $w$ has a neighbor $x$ in $H'$ such that $x\in R_t(H')$ and $x\not \in R_t(H)$. Since $x \not\in R_t(H)$, vertex $x$ must have been treated at some time-step before $t$.  Consider the last time-step $j$ for which $x$ was treated, in both protocols, before time-step $t$. Then because $x$ is red at time $t$ under $J'$, then $x$ must be compromised between time-steps $j$ and $t$. By the minimality of $t$ and the choice of $w$, $x$ must turn yellow at the same time under both $J$ and $J'$. Since $x$ is not treated between time-steps $j$ and $t$, vertex $x$ turns red at the same time in both protocols, contradicting $x\in R_t(H')$ and $x\not \in R_t(H)$.
\end{proof}

 An immediate consequence of Theorem~\ref{subgraph-thm} is that the $(r,s)$-treatment number of a graph is at least as large as the $(r,s)$-treatment number of each of its subgraphs.

\begin{cor} \label{subgraph-cor} If $H'$ is a subgraph of $H$, then $\tau_{r,s}(H')\leq \tau_{r,s}(H)$.
\end{cor}

\section{Treatment number and pathwidth}\label{sec:results}

In this section, we use the notion of path decompositions to bound our parameter from above, and then go on to make use of two particular types of protocols that we will call monotone and cautious.

\subsection{An upper bound for the treatment number using pathwidth}\label{subsec:pw}

As the example in Figure~\ref{fig-caterpillar} demonstrates, if a graph has a path-like  ordering of vertices, this ordering can be used to construct a protocol.  The next definition of path decomposition formalizes this notion and Theorem~\ref{monotone-upper} 
uses the resulting pathwidth to provide an upper bound for the $(r,s)$-treatment number.

\begin{defn} \rm
A \emph{path decomposition} of a graph $H$ is a sequence $(B_1,\dots,B_m)$ where each $B_i$ is a subset of $V(H)$ such that (i) if $xy \in E(H)$ then there exists an $i$, for which $x,y \in B_i$ and (ii)
for all $w \in V(H)$, if $w \in B_i \cap B_k$  and $i \leq j \leq k$ then $w \in B_j$. 

\label{path-decomp-def}
\end{defn}

The subsets $B_1,\dots,B_m$ are often referred to as \textit{bags} of vertices.  We can think of the sequence $(B_1,\dots,B_m)$ as forming the vertices of a path $P$.  Conditions (i) and (ii) ensure that if two vertices are adjacent in $H$, they are in some bag together and that the bags containing any particular vertex of $H$ induce a subpath of $P$. The graph $K_{1,3}$, shown in Figure~\ref{fig:K13}, has a path decomposition with bags $B_1=\{a_1,x\}$, $B_2=\{a_2,x\}$, $B_3=\{a_3,x\}$. The width of a path decomposition  and the pathwidth of a graph are defined below in such a way that the pathwidth  of a path is 1.

\begin{defn} \rm
The \emph{width} of a path decomposition $(B_1,\dots,B_m)$ is $\max_{1 \leq i \leq m} |B_i|-1$ and the \emph{pathwidth} of a graph $H$, denoted $pw(H)$, is the minimum width taken over all path decompositions of $H$.
\label{pathwidth-def}
\end{defn}

The next example illustrates the proof technique for  Theorem~\ref{monotone-upper}.

\begin{exa}\label{ex4x4} 
{\rm  We provide a $(1,2)$-protocol for the Cartesian product $P_4 \square P_4$ based on a minimum width path decomposition.  It is well known that $pw(P_4 \square P_4)=4$; for example, see \cite{Bo98}.  Since our protocol will have width 2 and  $\lceil (1+4)/(1+2)\rceil = 2$, our protocol will  achieve the upper bound in Theorem~\ref{monotone-upper}.  
Figure~\ref{fig:4x4} illustrates the bags in such a path decomposition and each bag is divided into $3$ sections, where the sections of bag $B_i$ are: section 1, section 2, section 3.  Observe that each vertex that appears in multiple bags always occurs in the same section of those bags.  

Starting at bag $B_1$, create a $(1,2)$-protocol where, in a given bag $B_i$, we treat vertices in section 1,  2,  3, then section 1, 2, 3 again; after this, we move to bag $B_{i+1}$ and repeat the process.  Thus, in our example $A_1 = A_4 = \{v_1,v_2\}$,  $A_2 = A_5 = \{v_3,v_4\}$, $A_3 = A_6 = \{v_5\}$, $A_7 = A_{10} = \{v_2,v_6\}$, $A_8 = A_{11} = \{v_3,v_4\}$, $A_9 = A_{12} = \{v_5\}$, and so on.  One can verify this leads to a  $(1,2)$-protocol that clears $P_4 \square P_4$.  In particular, observe that  for each vertex $v$,  if we consider the time period between the first and last time-step $v$ is treated, it is treated exactly every $r+s$ time-steps. By Observation~\ref{non-comprom-time-obs}, vertex $v$ does not turn red during this time period and indeed in this protocol, once a vertex is green, it never again becomes red.
}
\end{exa}

\begin{figure}[htbp]
\[ 
\includegraphics[width=\textwidth]{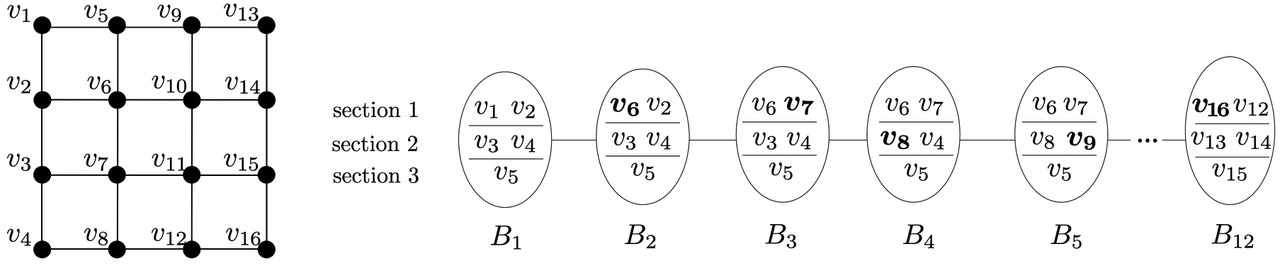}
\]
\caption{A minimum width path decomposition of $P_4 \square P_4$. The bold vertices are replacements as we move from left to right in the path.}

\label{fig:4x4}
\end{figure}

\begin{thm} \rm  If $H$ is a graph then $\tau_{r,s}(H) \leq \lceil \frac{1+pw(H)}{r+s}\rceil.$
\label{monotone-upper} 
\end{thm}

\begin{proof} Let $(B_1,\dots,B_m)$ be a path decomposition of $H$ whose width is $pw(H)$. By Definition~\ref{pathwidth-def} we know that  $|B_i| \leq 1+pw(H)$ for $1 \leq i \leq m$. For each $k$, we will partition $B_k$ into $r+s$ sections, $sct(k, 1), sct(k,2), \ldots, sct(k,r+s)$, some of which may be empty. 
We use these sections to create an $(r,s)$-protocol $(A_1,\dots,A_N)$ for graph $H$.  
In this protocol, we will treat vertices one bag at a time.  For each bag, we twice cycle through its sections, as in Example~\ref{ex4x4}.  Each partition of $B_k$ into $r+s$ sections will lead to $2(r+s)$ time-steps in the protocol.

Partition the vertices of $B_1$ arbitrarily into $sct(1,1), sct(1,2), \ldots, sct(1,r+s)$ so that each section has at most $\lceil \frac{1+pw(H)}{r+s}\rceil$ vertices.
Let $A_j=A_{(r+s)+j}=sct(1,j)$ for $1\leq j\leq r+s$. 
By construction, at the end of time-step $r+s$, all vertices in $B_1$ are green or yellow.  At the end of time-step $2(r+s)$, we have the additional property that any vertex in $B_1$, all of whose neighbors are also in $B_1$, is green.  Thus, at the end of time-step $2(r+s)$, any yellow vertex in $B_1$ has a red neighbor in $B_2$.

Partition $B_2$ into $r+s$ sets $sct(2,1), sct (2,2), \dots, sct(2,r+s)$ 
where $B_2\cap sct(1,i)\subseteq sct(2,i)$ 
for $1 \leq i \leq r+s$, that is, vertices of $B_2$ that are in section $i$ of $B_1$ are placed in section $i$ of $B_2$.  Distribute the vertices in $B_2 - B_1$ arbitrarily into $sct(2,1), sct (2,2), \dots, sct(2,r+s)$ 
so that each set has at most $\lceil \frac{1+pw(H)}{r+s}\rceil$ vertices. 
Let $A_{2(r+s)+1} = sct(2,1)$, $A_{2(r+s)+2}=sct(2,2)$, $\dots$, $A_{3(r+s)} = sct(2,r+s)$ (the first cycle of treatments for $B_2$) and $A_{3(r+s)+i} = A_{2(r+s)+i}$ for $1 \leq i \leq r+s$ (the second cycle of treatments for $B_2$).

At the end of time-step $4(r+s)$ all vertices in $B_1 \cup B_2$ are green or yellow and any yellow vertices have a red neighbor in $B_3$.

We continue by induction.  Suppose we have considered $B_1,\dots,B_k$ and extended our protocol so that at the end of time-step $2k(r+s)$ all vertices in $B_1 \cup \dots \cup B_k$ are green or yellow and any yellow vertices have a red neighbor in $B_{k+1}$.  We now partition $B_{k+1}$ into $r+s$ sets $sct(k+1,1), sct(k+1,2),\dots, sct(k+1,r+s)$ 
so that $B_{k+1} \cap sct(k,i) \subseteq sct(k+1,i)$ for $1 \leq i \leq r+s$. Distribute the vertices in $B_{k+1} - B_k$ arbitrarily into $sct(k+1,1), sct(k+1,2),\dots, sct(k+1,r+s)$
so that each set again has at most $\lceil \frac{1+pw(H)}{r+s}\rceil$ vertices. 
Let $A_{2k(r+s)+1} = sct(k+1,1) $, $A_{2k(r+s)+2}=sct(k+1,2)$, $\dots$, $A_{2k(r+s)} = sct(k+1,r+s)$; and $A_{(2k+1)(r+s)+i} = A_{(2k)(r+s)+i}$ for $1 \leq i \leq r+s$. 

At the end of time-step $(2k+2)(r+s)$ all vertices in $B_1 \cup \dots \cup B_{k+1}$ are green or yellow and any yellow vertices have a red neighbor in $B_{k+2}$, provided $B_{k+2}$ exists.  Since $B_m$ is the highest-indexed bag, at time-step $2m(r+s)$, all vertices in $B_1 \cup \dots \cup B_m$ are green.  Therefore, our protocol $(A_1,\dots,A_N)$ clears graph $H$ with $N = 2m(r+s)$ and furthermore, the $|A_i| \leq \lceil (1+pw(H))/(r+s)\rceil$.\end{proof}

\subsection{Nested sets of compromised vertices}
\label{subsec:mmc}

In this section we consider two special types of protocols:  monotone (Definition~\ref{monot-d}) and cautious (Definition~\ref{cautious-def}).
In a monotone protocol, once a vertex is treated  it never again becomes  compromised.  Thus, over time, the sets of compromised vertices are nested; at each time-step the current set of compromised vertices is a subset of those at the previous time-step. Inspired by Observation~\ref{non-comprom-time-obs},  in Definition~\ref{cautious-def} we define an $(r,s)$-protocol to be \emph{cautious} if for each vertex $v$, in between the first and last time-steps in which $v$ is treated, it is treated at least every $r+s$ time-steps. 

\begin{defn} \rm An $(r,s)$-protocol  $(A_1, A_2, \ldots, A_N)$ that clears graph $H$  is \emph{monotone}  if  it satisfies the following:  for all $v \in V(H)$, if $v \in A_j$ then $v \not\in R_i$ for all $i > j$.
\label{monot-d}
\end{defn}

Bernshteyn and Lee~\cite{BL21} define a protocol to be monotone if green vertices never become red again; however, in their model, there are no yellow vertices.  Our model combines yellow and green vertices in the definition of monotone.

The next proposition shows that we can equivalently define a protocol to be monotone if vertices that are about to be compromised are treated.  

\begin{prop} Let  $J$ be an $(r,s)$-protocol for graph $H$ where  $ J = (A_1, A_2, \ldots, A_N)$.  Then $J$  is monotone if and only if for every $j: 1 \le j \le N$ and every $v \in Y_j^1$, we have $v \in A_{j+1}$.
\label{heedful-prop}
\end{prop}

\begin{proof}
First suppose that $J$ is a monotone $(r,s)$-protocol for graph $H$.  Since all vertices are initially red, if $v \in Y_j^1$ then $v \in A_i$ for some $i: 1 \leq i <j$.  If in addition $v \not\in A_{j+1}$, then $v \in R_{j+1}$, contradicting the definition of monotone.

Conversely, in protocol $J$, if every vertex in $Y_j^1$ is also in $A_{j+1}$ then, once a vertex is treated, it never becomes red. Hence, $J$ is monotone.\end{proof}

In Observation~\ref{non-comprom-time-obs} we noted that 
 a vertex that is  treated at time-step $j$ cannot become compromised until $r+s$ time-steps later.  Thus we can achieve a monotone protocol by treating vertices frequently.  We define what it means for a protocol to be \emph{cautious} in Definition~\ref{cautious-def} and in Theorem~\ref{cautious-implies-monotone} prove that cautious protocols are monotone. Note that the protocols for graphs $H$ and $H'$ in Figure~\ref{fig-caterpillar} are cautious and monotone.

\begin{defn} \rm 
Let $H$ be a graph and suppose that $(A_1, A_2, \ldots, A_N)$ is an $(r,s)$-protocol that clears $H$. We say that this $(r,s)$-protocol is \emph{cautious} if the following holds for all vertices $w \in V(H)$: if the first occurrence of $w$ in a treatment set is in $A_j$ and the last is in $A_k$, then among any $r+s$ consecutive elements of the sequence $A_j, A_{j+1}, \dots,A_k$, at least one must contain $w$.
\label{cautious-def}
\end{defn} 

\begin{thm}
Any cautious $(r,s)$-protocol for graph $H$ is monotone.
\label{cautious-implies-monotone}
\end{thm}

\begin{proof}
Let $J$ be a cautious $(r,s)$-protocol for graph $H$  with $J = (A_1, A_2, \ldots, A_N)$.  By the definition of cautious we know $J$ clears $H$ and we will show that $J$ is monotone.    Consider $v \in V(H)$ and let the first occurrence of $v$ be in $A_j$ and the last in $A_k$.  By Observation~\ref{non-comprom-time-obs}, we know that $v$ is not red between time-steps $j$ and $k$ because $J$ is cautious.  Since $J$ clears $H$, we also know that $v$ is green at time-step $N$.  If $v$ were red after time-step $k$ then it would need to be treated after time-step $k$ in order to be green at time-step $N$, contradicting our choice of $k$.  Thus $J$ is monotone.
\end{proof}

Not all monotone protocols are cautious.  For example, the $(1,1)$-protocol $(\{a_1\}, \{x\}, \{a_2\},\{x\},\{a_3\},\{a_1\})$ 
for the graph $ K_{1,3}$  as labeled in Figure~\ref{fig:K13} is monotone but not cautious, because $a_1$ is treated at the first time-step and then again 5 time-steps later; however $r+s=2$. As previously discussed, $J'$ is not minimal; that is, it is 
unnecessary to treat vertex $a_1$ in time-step 6.  In Theorem~\ref{minimal-monotone-implies-cautious}  we show that 
protocols that are both monotone and minimal are always cautious.

\begin{thm} If $J$ is a minimal and monotone $(r,s)$-protocol for graph $H$,  then $J$ is cautious.
\label{minimal-monotone-implies-cautious}
\end{thm}

\begin{proof}
For a proof by contradiction, suppose $J$ is not cautious. Then $\exists w\in V(H)$ with first occurrence in $A_j$ and last in $A_k$ such that 
$w\not\in A_t\cup A_{t+1}\cup \cdots \cup A_{t+r+s-1}$ and $w\in A_{t+r+s}$, with $j<t$ and $t+r+s-1< k$. 
Because $J$ is monotone, $w$ is never red after time-step $j$, and therefore it is green or yellow at time-step $t+r+s-1$. 
Now $w$ was last treated at or before time-step $t-1$, so if $w$ is green at time $t+r+s-1$, then  $w\in G^1_{t+r+s-1}$. 
Otherwise, $w\in  Y_{t+r+s-1}$. We consider these cases separately. 

If $w\in G^1_{t+r+s-1}$, since $w\in A_{t+r+s}$, then by definition of minimality, $w$ has a neighbor $y\in R_{t+r+s} \cup R_{t+r+s+1}\cup \cdots R_{t+r+s+r-1}$. By monotonicity, vertex $y$ has been red from time-step 1 at least until time-step $t+r+s$. In particular, $y$ is red in the time-steps $t, t+1, t+2, \ldots, t+r+s-1$. Vertex $w$ was last treated at or before time-step $t-1$, so by time-step $t+r-1$, it has no protection remaining, and $w$ becomes yellow due to its red neighbor $y$.
Thus, vertex $w$ turns red $s$ time-steps later, and is red at time-step $t+r+s-1$, because it is not treated in that time period. This contradicts $w\in G^1_{t+r+s-1}$.   

Otherwise, $w\in Y_{t+r+s-1}$. In this case, $w$ turned from green to yellow at most $s$ time-steps previously, 
so it became red at time $t+r$ or later, because it remains yellow for at most $s$ time-steps. Vertex $w$ was last treated at or before time-step $t-1$, and its protective period lasted until time-step $t+r-2$ or sooner. Thus, vertex $w$ turns to yellow due to its red neighbor $y$ 
at time-step $t+r-1$, and thus became red $s$ time-steps later, and is red at time-step $t+r+s-1$. This contradicts $w\in Y_{t+r+s-1}$.\end{proof}

While cautious protocols are monotone (Theorem~\ref{cautious-implies-monotone}), they are not necessarily minimal. For example, the protocol
$J''=(\{a_1\}, \{x\}, \{a_2\},\{x\},\{a_3\},\{a_3\})$  for the graph $K_{1,3}$, labeled as in Figure~\ref{fig:K13},
is cautious but not minimal.
 In Theorem~\ref{upper-bound-thm}, we show that using a cautious protocol $J$ to clear $H$ yields a path decomposition of $H$, and that $width(J) \ge  \lceil (\mbox{$1+$ pathwidth of $H$})/2\rceil$.  Together with Theorem~\ref{minimal-monotone-implies-cautious}, we get the same upper bound for monotone  protocols that are minimal.

\begin{thm} \rm If $J$ is a cautious $(r,s)$-protocol that clears graph $H$, then $width(J) \geq \lceil \frac{1+pw(H)}{r+s} \rceil$.\label{upper-bound-thm} \end{thm}

\begin{proof} Let $J = (A_1,A_2,\dots,A_N)$ and let $m = N-(r+s)+1$.   We will create a path decomposition of $H$ with bags $B_1,B_2,\dots,B_m$ where each $B_i$ is the union of $r+s$ treatment sets, and thus $|B_i| \leq (r+s)\cdot width(J)$ for $1 \leq i \leq m$.  Thus, \[width(J) \geq \frac{ \max_{1 \leq i \leq m}|B_i|}{r+s} \geq \frac{1+pw(H)}{r+s}.\] Since the width of $J$ is an integer, it suffices to construct the path decomposition. The path decomposition construction in this proof is illustrated below in Example~\ref{path-decomp-ex}.

Define the bags $B_i$ as $B_i = A_i \cup A_{i+1} \cup \dots \cup A_{i+(r+s-1)}$ for $1 \leq i \leq m$.  Note that $A_i$ first appears in $B_{i-(r+s-1)}$ and last appears in $B_i$ (or $B_1$ if $i -(r+s-1)<0$).  It remains to show the bags $B_1,\dots,B_m$ form a path decomposition.

To prove (ii) of Definition~\ref{path-decomp-def}, let $w$ be a vertex in $H$ and let $A_j$ be the first element of $J$ that contains $w$, and $A_k$ be the last element of $J$ that contains $w$.  As noted above, this means $w \in B_{j-(r+s-1)}$ and $w \in B_k$.  Because $J$ is cautious, there is no sequence of $r+s$ consecutive time-steps between $j$ and $k$ in which $w$ is not treated; so $w$ is in every bag from $B_{j-(r+s-1)}$ to $B_k$ (again, or from $B_1$ if $j-(r+s-1)<0$).

Now to prove (i) of  Definition~\ref{path-decomp-def}, 
suppose for a contradiction  that $uw$ is an edge in $H$ and no bag contains $uw$. Without loss of generality, suppose that the first occurrence of $u$ is earlier than the first occurrence of $w$ in $J$. Let the first occurrence of $u$ be in $A_j$ and the last in $A_k$. Then $u$ is in bags $B_{j-(r+s-1)},B_{j-(r+s-2)}, \dots, B_k$ and note that $B_k = A_k \cup A_{k+1} \cup \dots \cup A_{k+(r+s-1)}$.  The lowest-indexed bag that could contain $w$ is $B_{k+1}$, and thus the lowest-indexed treatment set that can contain $w$ is $A_{k+(r+s)}$. Since $u$ was last treated during time-step $k$, at time-step $k+(r+s-1)$, it would have turned yellow due to its red neighbor $w$,
contradicting the fact that the treatment protocol clears $H$.  
\end{proof}

\begin{exa} \rm 
To illustrate the construction  given in the proof of Theorem~\ref{upper-bound-thm}, we start with 
the graph $H$ and a $(1,1)$-protocol that clears it, both of which are shown in Figure~\ref{fig-caterpillar}.  The construction  results in  the path decomposition of $H$ that has the following $12$ bags:  $B_1 = \{a_1,x\}$,
 $B_2 = B_3 = \{a_2,x\}$, $B_4 = B_5 = \{a_3,x\}$, $B_6 = \{x,y\}$, $B_7 = B_8 = \{b_1,y\}$, $B_9 = \{y,z\}$,  $B_{10} = B_{11} = \{c_1,z\}$, $B_{12} = \{c_2,z\}$.
\label{path-decomp-ex}
\end{exa}

We conclude this section with a corollary for graphs $H$ for which there is strict inequality in Theorem~\ref{monotone-upper}.

\begin{cor}\label{cor:cautious-not} Let $H$ be a graph such that $\tau_{r,s}(H)<  \left \lceil \frac{1+pw(H)}{r+s}\right \rceil$. 
 The only $(r,s)$-protocols that clear $H$ and have width $\tau_{r,s}(H)$ are non-cautious $(r,s)$-protocols.
\end{cor} 

\begin{proof}
By Theorem~\ref{upper-bound-thm}, every cautious $(r,s)$-protocol for $H$ will have width at least   $\tau_{r,s}(H)$.  Thus,  the only way to find an $(r,s)$-protocol that achieves  $\tau_{r,s}(H)<\left \lceil \frac{1+pw(H)}{r+s}\right \rceil$ is to use a non-cautious $(r,s)$-protocol.  
\end{proof}

\section{The treatment number when $r=s=1$}\label{sec:r=s=1}

In the next sections, we provide  results for the case where $r=s=1$. For a graph $H$, we  let  $\tau_{1,1}(H)=\tau(H)$ for the rest of the paper and abbreviate $(1,1)$-protocol by \emph{protocol} and abbreviate $(r,s)$-treatment number by \emph{treatment number}.  For a subset $S \subseteq V(H)$, we define the neighborhood of $S$ to be $N(S) = \{v \in N(x) | x \in S\}$ and note that some vertices of $S$ may be in $N(S)$.

\subsection{Characterization of graphs with treatment number 1}\label{subsec:char1}

If a protocol of width 1 clears  a graph, then  the number of green vertices can be increased from one time-step to the next only when there is at most one   non-green vertex 
that is adjacent to a green  vertex.  This leads to the following necessary condition for graphs $H$ with $\tau(H) = 1$.

\begin{lem}
If $H$ is a graph with $\tau(H) = 1$ then, for every $p$ with $1 \le p \le |V(H)| -1$, there exists $S \subseteq V(H)$ with $|S| = p$ and $| N(S) - S| \le 1$.
\label{treatment-one-lemma}
\end{lem}

\begin{proof}
Let $H$ be a graph with $\tau(H) = 1$ and let $(A_1,A_2,  \ldots , A_N)$ be a protocol that clears $H$ with $|A_i| \le 1$ for each $i$.  
The only way for a vertex to become green  is to be treated, so the number of green vertices increases by at most $1$ at each time-step. This protocol clears $H$, so for $p$ with $1 \le p \le |V(G)| -1$ there is a time-step $j$ in which the number of green vertices increases from $p$ to $p+1$.  
The set $G_j$  consists  of  the  $p$ green vertices at time-step $j$, and they must remain green at time-step $j+1$ while a new vertex is treated.   If there were two vertices in the set  $N(G_j) - G_j$    then at most one of these could be treated at time-step $j+1$ and the other would compromise a vertex of $G_j$, a contradiction.  Thus $| N(G_j)- G_j| \le 1$ and the set $G_j$ is our desired set $S$.\end{proof}

Let $K'_{1,3}$ be the graph obtained by subdividing each edge of $K_{1,3}$ exactly once and let $C_n$ be a cycle on $n \geq 3$ vertices.
We next use Lemma~\ref{treatment-one-lemma} to determine the treatment number of  $C_n$ and $K_{1,3}'$ and then use these results to characterize graphs with treatment number $1$. 
Every vertex of a cycle has two neighbors, so the conclusion of Lemma~\ref{treatment-one-lemma} fails when $p=1$.   Thus $\tau(C_n) \ge 2$.  The following protocol shows that $\tau(C_n) \le 2$.  First treat two adjacent vertices, and at each subsequent time-step, treat the compromised neighbors of the vertices most recently treated.  An alternate proof is to use Theorem~\ref{monotone-upper} and that fact that the pathwidth of $C_n$ is $2$; see \cite{Bo98}. We record this in Observation~\ref{obs-cycle-small-spider}.

In $K'_{1,3}$, every set $S$ of $3$ vertices has $|N(S) - S| \ge 2$.  Therefore, $\tau(K'_{1,3}) \ge 2$ by Lemma~\ref{treatment-one-lemma}.  It is an easy exercise to find a  protocol of width $2$ that clears $K'_{1,3}$. Thus $\tau(K'_{1,3}) = 2$, which we  also record in Observation~\ref{obs-cycle-small-spider}.

 \begin{obs}\label{obs-cycle-small-spider} The treatment number of any cycle is $2$ and the treatment number of $K'_{1,3}$ is $2$. \end{obs}

We now characterize those graphs $H$ for which $\tau(H)=1$.  A \emph{caterpillar} is a graph that consists of a path together with  zero or more leaves incident to each vertex of the path.
 
\begin{thm} For a graph $H$ that contains at least one edge, $\tau(H)=1$ if and only if $H$ is a caterpillar.
\label{one-caterpillar}
 \end{thm}

\begin{proof} Let $H$ be a caterpillar. Note that the pathwidth of a caterpillar is one; see \cite{Bo98}.  Now Theorem~\ref{monotone-upper} (with $r=s=1$) implies that  $\tau(H)\leq 1$.  Since $H$ contains at least one edge, $\tau(H)\geq 1$.  Thus, $\tau(H)=1$.

For the other direction, we assume $H$ is a graph that contains at least one edge and $\tau(H)=1$. It is well known  that caterpillars are precisely those trees that have no induced $K'_{1,3}$.  By Observation~\ref{obs-cycle-small-spider}, cycles have treatment number $2$, so by Corollary~\ref{subgraph-cor}, $H$ cannot contain a cycle as a subgraph and hence $H$ is a tree.  Similarly, by Observation~\ref{obs-cycle-small-spider},  the graph $K'_{1,3}$ has treatment number 2, and so it cannot be a subgraph of $H$, and hence $H$ is a caterpillar.
\end{proof}

\subsection{A tool for finding lower bounds of $\tau(H)$}

We generalize Lemma~\ref{treatment-one-lemma} that helped characterize
graphs with  treatment number 1.
For any protocol that clears graph $H$, the number of green vertices must
increase beyond each threshold from 2 to $|V(H)|-1$. Every time there is an
increase, some set $S$ of vertices will remain green without being
treated. The number of non-green   
neighbors of vertices in $S$ is
restricted based on the number of treatments available at the same time-step.
 
\begin{thm}
Let $H$ be a graph.  For each $p: \tau(H) + 1 \le p \le |V(H)| -1$ there exists a set $S \subseteq V(H)$ with $p - \tau(H) + 1 \le |S| \le p$ so that $| N(S) - S|  \le  2 \, \tau(H) -1$.
\label{necessary-condition}
\end{thm}

\begin{proof}
Let      $(A_1,A_2,  \ldots , A_N)$ be a protocol that clears graph  $H$ with  $|A_i| \le \tau(H)$ for each $i$.  Thus  $G_N = V(H)$; that is, all vertices are green at time-step $N$.   For any $p$ with $\tau(H) + 1  \le p \le |V(H)| - 1$, there must be a time-step $j$ for which $|G_j| \le p$ and $|G_{j+1}| \ge p+1$. The only way for a vertex to become green is to be treated, so the number of green vertices increases by at most $\tau(H)$ at each time-step and hence $|G_j| \ge 
|G_{j+1}| - \tau(H) \ge p+1 - \tau(H)$.

We partition $G_j$ into three parts: $G_j \cap Y_{j+1}$ (vertices in $G_j$ that become vulnerable at time-step $j+1
$), $G_j \cap A_{j+1}$ (vertices  in $G_j$ that are treated at time-step $j+1$) and the set $S$ of remaining vertices (vertices of $G_j$ that remain green at time-step $j+1$ without being re-treated). 
In addition, let $a = |G_j \cap Y_{j+1}|$, let $b = |G_j \cap A_{j+1}|$, and let   $k = |A_{j+1} \cap \overline{G_j}|$, the number of vertices outside of $G_j$ that are treated at time-step $j+1$.

First we show  the condition $p  - \tau(H) + 1 \le |S| \le p$.  The upper bound follows immediately from  $|S| \le |G_j| \le p$.  For the lower bound, since  $|G_j| \ge  p+1 - \tau(H)$, we can write
 $|G_j| = p-d$ for some  integer $d$ with $0 \le d \le \tau(H) -1$.  Since there are $k$ vertices outside of $G_j$ treated at time-step $j+1$ and $a$ vertices of $G_j$ that become  yellow at time-step $j+1$, we have $|G_{j+1}| = |G_j| + k-a$.  Recall that $|G_{j+1}| \ge p+1$, and hence $k-a = |G_{j+1}| - |G_j| \ge (p+1) - (p-d) = d+1$.  We conclude that $d+a \le k-1$ or equivalently $-d-a \ge -k+1$.  Also note that there are exactly $b+k$ vertices treated at time-step $j+1$, so $b+k \le \tau(H)$, or equivalently, $-b-k \ge -\tau(H)$.   Now we complete the proof of the lower bound on $|S|$ as follows:
 $$|S| = |G_j| - b - a = p - d - b - a = p - b  + (-d -a) \ge p - b + (-k + 1) = p + 1 + (-b-k) \ge p+1 -\tau(H).$$

\medskip

It remains to show the inequality $| N(S) - S|  \le  2 \, \tau(H) -1$.  
By construction,  $|G_j| = |S| + a + b$.  
By definition, the vertices of $S$ remain green at time-step $j+1$ so their only neighbors can be other vertices in $G_j$ and the $k$ vertices outside of $G_j$ that are treated at time-step $j+1$, because all other vertices are compromised.  Thus $| N(S) - S| \le a + b  + k$.  However, $k+b = |A_{j+1}| \le \tau(H)$.  Combining these we obtain the following: $|N(S) - S| \le  a + b  + k \le a + \tau(H).$ 

Recall that $|G_{j+1}| \ge p + 1 \ge |G_j| + 1$; that is, there are more green vertices at time-step $j+1$ than at time-step $j$. There are $a$ newly yellow vertices, so there must be more than $a$ vertices that are treated at time-step $j+1$, and thus $ |A_{j+1}| \ge a+1$.  Therefore, $\tau(H) \ge|A_{j+1}| \ge a+1$ and we obtain 
$| N(S)- S| \le a + \tau(H) \le (\tau(H) - 1) + \tau(H)   = 2\tau(H) -1$, as desired.
\end{proof}

As a special case, if $H$ is a graph with $\tau(H) = 2$ then for every $p$ with $2 \le p \le |V(H)| -1$ there exists $S \subseteq V(H)$ with $|S| = p-1$ or $p$ and $| N(S)-S| \leq 3$.  The next two results use this special case of Theorem~\ref{necessary-condition}.

Recall that the Petersen graph can be defined as the graph whose vertices are  the $2$-element subsets of a $5$-element set and where two vertices are adjacent precisely when their corresponding subsets are disjoint.

 \begin{prop}\label{prop:peterson} 
 The Petersen graph has treatment number $3$.
\end{prop}

\begin{proof} 
Let $H$ be the Petersen graph and  let $V(H)$ be the set of $2$-element subsets of $\{1,2,3,4,5\}$.  
  One can check that the following treatment protocol clears $H$:
$A_1 = \{12, 34, 35\}$, \ $A_2 = \{15, 25, 45\}$, \ $A_3 = \{13, 23, 35\}$, \ $A_4 = \{14, 24, 25\}$, \ $A_5 = \{15, 23, 34\}$.   Thus
 $\tau(H) \le 3$. 

It remains to show $\tau(H) > 2$.   We know that $\tau(H) \neq 1$ by Theorem~\ref{one-caterpillar}. Suppose for a contradiction that $\tau(H) = 2$.  By Theorem~\ref{necessary-condition}, we know that for each $p:3 \le p \le 9$, there exists $S\subseteq V(H)$ with $p-1 \le |S| \le p$ so that $|N(S) - S|  \le  3$.  Consider $p = 3$.  For all $S \subseteq V(H)$ with 
$|S| = p-1 = 2$ one can check that $| N(S) - S| \ge 4 >  3$.  Likewise, for all $S \subseteq V(H)$ with 
$|S| = p = 3$ one can check that $| N(S) - S| \ge 5 >  3$.  This is a contradiction.\end{proof}

For a set $S$, the set $N(S)-S$ represents the \emph{boundary}  of $S$; that is, the expression represents the neighbors of $S$ that are not in $S$.  This quantity has also been studied in the context of the vertex isoperimetric value of a graph; see the surveys~\cite{bezrukov,leader}. We combine 
Theorem~\ref{necessary-condition} with the known isoperimetric value of the grid to find the treatment number of the grid.

\begin{defn}
Let $H$ be a graph. 
For $k \in \{1,2,\dots,|V(H)|\}$, the vertex isoperimetric value of graph $H$ with integer $k$ is denoted $\Phi_V(H,k)$, and is defined as $$\Phi_V(H,k) = \min_{S \subseteq V(H): |S|=k} |N(S)-S|.$$ 
\end{defn}
Consider, for example, $P_4 \square P_4$, as shown in Figure~\ref{fig:4x4}. The interested reader can verify that for any set $S$ of eight vertices, there are at least four vertices not in $S$ that have a neighbor in $S$; that is, $|N(S)-S| \geq 4$.  Since we can find a set of eight vertices which have exactly four neighbors not in the set, $\Phi_V(P_4 \square P_4,8)=4$.

\begin{thm} For $n \geq 3$, $\tau(P_n \square P_n) = \lceil \frac{n+1}{2}\rceil$.
\label{i-grids}
\end{thm}

\begin{proof} Assume $n \geq 3$.  It is well known that $pw(P_n \square P_n)=n$ (see~\cite{Bo98}).  By Theorem~\ref{monotone-upper}, $\tau(P_n \square P_n) \leq \lceil \frac{n+1}{2}\rceil$.  Thus, $\tau(P_n \square P_n) = \lceil \frac{n+1}{2}\rceil - \ell$ for some integer $\ell \in \{0,1,\dots,\lceil \frac{n-1}{2}\rceil\}$.\medskip

Let $p = \frac{n^2+n-2}{2}$ and observe that $\lceil \frac{n+1}{2}\rceil - \ell + 1 \leq p \leq n^2-1$ because $n \geq 3$.  Thus, by Theorem~\ref{necessary-condition}, there exists a set $S \subseteq V(P_n \square P_n)$ with 
\[p-\left(\left \lceil \frac{n+1}{2}\right \rceil-\ell\right)+1 \leq \;|S| \;\leq p\]
so that $|N(S)-S| \leq 2(\lceil \frac{n+1}{2}\rceil-\ell)-1$.  We simplify the latter inequality: \begin{equation}\label{eqn:NS} |N(S)-S| \leq 2\Big\lceil \frac{n+1}{2}\Big\rceil-2\ell-1 = \begin{cases} n-2\ell & \text{ if $n$ is odd,} \\ n+1-2\ell & \text{ if $n$ is even.}\end{cases}\end{equation}

To complete the proof, we use the following result, proven in~\cite{bela1} and explicitly stated in~\cite{one-vis}: $$\Phi_V(P_n \square P_n,k) = n \text{~~for~~} \frac{n^2-3n+4}{2} \leq k \leq \frac{n^2+n-2}{2}.$$  Consequently, $$|N(S)-S| \geq n \text{~~for~~} \frac{n^2-3n+4}{2} \leq |S| \leq \frac{n^2+n-2}{2}.$$ But this contradicts the inequality in~(\ref{eqn:NS}) unless $\ell=0$.  Thus, for $n \geq 3$, $\tau(P_n \square P_n) = \lceil \frac{n+1}{2}\rceil$.
\end{proof}

\section{Subdivisions when $r=s=1$}\label{sec:subdivisions}

The main result of this section is proving that every tree has a subdivision that has treatment number at most 2. We begin with $P_4 \square P_4$, to illustrate how subdividing an edge can change the treatment number. 
By Theorem~\ref{i-grids}, $\tau(P_4 \square P_4)=3$. By subdividing one edge, of $P_4 \square P_4$, we show that the resulting graph has treatment number 2. 

\begin{exa}
There is a subdivision of $P_4 \square P_4$  with treatment number 2.  \end{exa}

\begin{proof}
Let  $ V = V(P_4 \square P_4) = \{a,b,c,d,e,f,g,h,i,j,k,\ell,m,n,o,p\}$, as in Figure~\ref{4disj}. Let $H$ be $P_4 \square P_4$ with the edge $\{e,i\}$ replaced by a path $P$ with 30 interior vertices, $w_1, w_2, \ldots, w_{30}$ from $i$ to $e$ (i.e., $w_{30}$ is adjacent to $e$).  We construct a treatment protocol for $H$ beginning with the interior vertices of $P$ as follows:
$A_1 = \{w_1,w_2\}$, $A_2 = \{w_3,w_4\}$,  \ldots, $A_{15} = \{w_{29},w_{30}\}$.

Then we clear the vertices of $V$  using  the treatment sets  $A_{16}$ through $A_{30}$ in the table below. This table also shows which of these vertices is green during  time-steps 16 through 30.

\vspace{.05in} 

\begin{tabular}{c|ccccccccccccccc}
$t$ & 16 & 17 & 18 & 19 & 20 & 21 & 22 & 23 & 24 & 25 & 26 & 27 & 28 & 29 & 30  \\ \hline 
$A_t$ & $e,f$ & $a,b$ & $c,f$ & $d,h$ & $c,f$ & $g,h$ & $f,\ell$ & $g, k$ & $f,\ell$ & $j,k$ & $\ell, p$ & $j,k$ & $p, o$ & $j, n$ & $i,m$ \\  \hline
$G_t \cap V$ & $e,f$ & $a,b$ & $c,f$ & $d,h$ & $c,f$ & $g,h$ & $f,\ell$ & $g, k$ & $f,\ell$ & $j,k$ & $\ell, p$ & $j,k$ & $p, o$ & $j, n$ & $i,m$ \\ 
 & & $e$  & $a,b$ & $a,b$  & $a,b$& $a,b$ & $a,b$ & $a,b$ & $a,b$ & $a,b$ & $a,b$ & $a,b$ &$a,b$ & $a,b$ & $a,b$ \\
  & &      & $e$ & $e$ & $d,e$ & $c,d$ & $c,d$ & $c,d$ & $c,d$ & $c,d$ & $c,d$ & $c,d$ & $c,d$ & $c,d$ & $c,d$ \\
    & &      &                      &  &  & $e$ & $e,h$ & $e,h$ & $e,g$ & $e,f$ & $e,f $& $e,f$& $e,f$ & $e,f$ & $e,f$ \\
        & &      &                      &  &  &  &               &  & $h$ & $g,h$  & $g,h$ & $g,h$ & $g,h$& $g,h$ & $g,h$ \\
                & &      &                      &  &  &  &               &  &                   &   & & $\ell$  &$k,\ell$ & $k, \ell$ & $j,k$\\
                                & &      &                      &  &  &  &               &  &                   &   & &     &     & $o,p$ & $\ell, n$\\
                                  & &      &                      &  &  &  &               &  &                   &   & &     &     & & $o,p$\\

\end{tabular}

\vspace{.05in}

Although $w_1$ and $w_2$ are green at  time-step 1, $w_1$ is adjacent to $i$ and is corrupted 
during time-step 2.  During time-step 3, $w_1$ turns red and corrupts  
$w_2$. Similarly, during time-step 4, $w_2$ turns red and corrupts 
$w_3$. Iteratively,  for $3\leq j\leq 30$, during time-step $j$,  vertex $w_{j-2}$ turns red and corrupts 
$w_{j-1}$.  Thus, during time-step $30$, $w_{28}$ turns red and corrupts 
$w_{29}$, while $w_{30}$ remains green at time-step $30$. The vertex $i$ is treated during time-step 30 and  all the vertices in $V$ remain green, since $e$ is still green. During time-step 31, we treat $w_{29}$ and $w_{28}$. During time-step 32, we treat $w_{27}$ and $i$. We proceed in the same manner to treat the rest of the interior vertices of $P$, moving towards $i$, and treating $i$ every other time-step to keep $i$ from corrupting  
any other vertex. When we complete the treatment of the vertices on $P$, every vertex in $H$ is green.\end{proof} 

\begin{figure}[hb] 

\begin{center}
\begin{tikzpicture}[scale=.9]

\draw[black](0,1)--(0,0)--(1,0);
\draw[black](0,2)--(0,3)--(1,3);
\draw[black](2,3)--(2,2)--(3,2);
\draw[black](2,0)--(2,1)--(3,1);

\draw[black](1,1)--(1,2);
\draw[black](1,1)--(0,1)--(0,2)--(1,2);
\draw[black](1,1)--(2,1)--(2,2)--(1,2);
\draw[black](1,1)--(1,0);
\draw[black](2,0)--(3,0)--(3,1)--(3,2)--(3,3)--(2,3)--(1,3)--(1,2);
%\draw[gray] (1,0) to [bend right=90] (2,0);
\draw[ultra thick,gray,decorate,decoration={coil,aspect=0,segment length=3mm}] (1,0) -- (2,0);

\filldraw[darkgray]
(0,0) circle [radius=3pt]
(1,0) circle [radius=3pt]
(2,0) circle [radius=3pt]
(3,0) circle [radius=3pt]

(0,1) circle [radius=3pt]
(1,1) circle [radius=3pt]
(2,1) circle [radius=3pt]
(3,1) circle [radius=3pt]

(0,2) circle [radius=3pt]
(1,2) circle [radius=3pt]
(2,2) circle [radius=3pt]
(3,2) circle [radius=3pt]

(0,3) circle [radius=3pt]
(1,3) circle [radius=3pt]
(2,3) circle [radius=3pt]
(3,3) circle [radius=3pt]
;

\node at (-.3,0.2) {$a$};
\node at (-.3,1.2) {$b$};
\node at (-.3,2.2) {$c$};
\node at (-.3,3.2) {$d$};

\node at (.7,0.2) {$e$};
\node at (.7,1.2) {$f$};
\node at (.7,2.2) {$g$};
\node at (.7,3.2) {$h$};

\node at (1.75,0.3) {$i$};
\node at (1.7,1.2) {$j$};
\node at (1.7,2.2) {$k$};
\node at (1.7,3.2) {$\ell$};

\node at (2.7,0.2) {$m$};
\node at (2.7,1.2) {$n$};
\node at (2.7,2.2) {$o$};
\node at (2.7,3.2) {$p$};

\end{tikzpicture}
\end{center}

\caption{The $4\times 4$ grid $P_4 \square P_4$ with one  edge  subdivided (indicated by a wavy line).} 

\label{4disj}

\end{figure}

We next provide a construction that puts together two subgraphs, each of which can  be cleared with 2 treatments at each time-step, to get a larger graph that can also be cleared with 2 treatments at each time-step.  We use this construction primarily for trees. Let $T_{m,d}$ be the complete $m$-ary tree of depth $d$ and $\hat{T}_{m,d}$ be the tree obtained by attaching a leaf to the root of $T_{m,d}$. We call this added leaf the \emph{stem} of $\hat{T}_{m,d}$.   Figure~\ref{fig-depth-3-tree} illustrates how to construct  a subdivision of $\hat{T}_{2,3}$ from subdivisions of two copies of $\hat{T}_{2,2}$. In our proof, we construct a protocol that can be used for the larger tree, using the protocols of the smaller trees. 

 \begin{thm} \label{binary}
For $1\leq k\leq 2$, let $H_k$ be a graph with a leaf  $\ell_k$. Let $J_k$ be a protocol of width 2 that clears $H_k$, takes $s_k$ time-steps, and for which   $\ell_k $ is treated only in the last time-step. Let $H$ be the graph that consists of $H_1+H_2$, a new path between $\ell_1$ and $\ell_2$ with $s_1$ new interior vertices, and a new leaf $\ell$ adjacent to $\ell_1$.  Then there is a protocol  of width 2 that clears $H$  in  at most $s_2+3s_1$ time-steps and for which  $\ell$ is treated only in the last time-step.
\end{thm}

\begin{proof} Let $P$ be the new path between $\ell_1$ and $\ell_2$ with $s_1$ new interior vertices.  Our protocol for clearing $H$ will first clear $H_2$, then the interior of  $P$ and $\ell_1$, then $H_1$, then  the vertices on $P$ that have changed color, and finally vertex $\ell$.  In this protocol, once the graph $H_2$ is cleared, its vertices will remain green. 
We begin by clearing $H_2$ in $s_2$ time-steps by following  protocol $J_2$, with $\ell_2$ treated only  in the last time-step. Since $\ell_2$ is cleared last, it is green at time-step $s_2$. Then clear the interior vertices of $P$ and $\ell_1$, in order from $\ell_2$ to $\ell_1$,  using $s_1$ time-steps. Note that we can clear these vertices with our 2 treatments per time-step in $\lceil\frac{s_1+1}{2}\rceil$ time-steps, but we count it as $s_1$ time-steps to simplify the calculation.  

Next, clear $H_1$ in $s_1$ time-steps by following  protocol $J_1$, with $\ell_1$ treated  only in the last time-step.   The neighbor of $\ell_1$ on $P$ turns yellow at time-step $s_2+s_1+2$, and red at time-step $s_2+s_1+3$. It changes the color of the vertex adjacent to it  on $P$ at  time-step $s_2+s_1+3$, and the process of 
color change to yellow and then red continues, one vertex at each time-step, in order from $\ell_1$ towards $\ell_2$.  We start to clear $H_1$ at time $s_2+s_1+1$ and finish at time $s_2+s_1+s_1$, with $\ell_1$ being treated in this last time-step. At most $s_1-1$ vertices on $P$ turn red or yellow.  Since $P$ has $s_1$ interior vertices, during the $s_1$ time steps of protocol $J_1$, vertex $\ell_2$ and its neighbor on $P$ remain green.
The last series of treatments are to treat the red and yellow vertices on $P$, starting at the end near $\ell_2$ and moving towards $\ell_1$, and treating $\ell_1$ every other time-step to prevent $\ell_1$ from becoming red. Because we have two treatments at each time-step, we can treat both $\ell_1$ and a vertex on $P$ in the same time-step.  Thus we can clear $P$  in  $s_1-1$ time steps, starting at time-step
$s_2+2s_1+1$, and finishing at time-step $s_2+2s_1+(s_1-1)$. We add one additional time-step to  treat $\ell$ and $\ell_1$. The total number of  time-steps is at most $s_2+3s_1$ and $\ell$ is treated only at the last time-step.
\end{proof} 

\begin{figure}[htb] \label{tree-fig} 
\begin{center}
\begin{tikzpicture}[scale=.8]
\draw[thick] (0,0)--(.5,1)--(1,0); 
\draw[thick] (2.5,-.5)--(3,.25)--(3.5,-.5); 
\draw[thick] (.5,1)--(2,2); 
\draw[thick] (2,2)--(2,3); 
\draw[thick] (3,.8)--(3,.25);
\draw[ultra thick,gray,decorate,decoration={coil,aspect=0,segment length=4mm}] (2,2) -- (3,.8);
\filldraw[blue]

(3,.8) circle [radius=3pt]

(2,2) circle [radius=3pt]
(2,3) circle [radius=3pt];

\filldraw[black]
(0,0) circle [radius=3pt]
(1,0) circle [radius=3pt]
(0.5,1) circle [radius=3pt]
(2.5,-.5) circle [radius=3pt]
(3,.25) circle [radius=3pt]
(3.5,-.5) circle [radius=3pt];

\node(ell) at (1.7,3.2) {$\ell$};
\node(ell1) at (1.7,2.3) {$\ell_1$};
\node(ell2) at (3.3,1.1) {$\ell_2$};

\node(a2) at (2,-2) {(a)};

\draw[thick] (7,0)--(7.5,1)--(8,0); 
\draw[thick] (10,0)--(10.5,1)--(11,0); 
\draw[thick] (7.5,1)--(9,2); 
\draw[thick] (9,2)--(9,4); 
\draw[thick] (10,1.34)--(10.5,1);
\draw[ultra thick,gray,decorate,decoration={coil,aspect=0,segment length=3mm}] (9,2) -- (10,1.34);

\draw[thick] (13,-1)--(13.5,0)--(14,-1); 
\draw[thick] (16,-1)--(16.5,0)--(17,-1); 
\draw[thick] (13.5,0)--(15,1); 
\draw[thick] (15,1)--(15,2); 
\draw[thick] (16,.34)--(16.5,0);
\draw[ultra thick,gray,decorate,decoration={coil,aspect=0,segment length=3mm}] (15,1) -- (16,.34);

\draw[ultra thick,gray,decorate,decoration={coil,aspect=0,segment length=6.3mm}] (9,3) -- (15,2);

\filldraw[blue]
(15,2) circle [radius=3pt]
(9,3) circle [radius=3pt]
(9,4) circle [radius=3pt];

\filldraw[black]
(10,1.34) circle [radius=3pt]

(7,0) circle [radius=3pt]
(8,0) circle [radius=3pt]
(7.5,1) circle [radius=3pt]
(10,0) circle [radius=3pt]
(10.5,1) circle [radius=3pt]
(11,0) circle [radius=3pt]
(9,2) circle [radius=3pt];

\filldraw[black]

(16,.34) circle [radius=3pt]

(13,-1) circle [radius=3pt]
(14,-1) circle [radius=3pt]
(13.5,0) circle [radius=3pt]
(16,-1) circle [radius=3pt]
(16.5,0) circle [radius=3pt]
(17,-1) circle [radius=3pt]
(15,1) circle [radius=3pt];

\node(ell) at (8.7,4.3) {$\ell$};
\node(ell1) at (8.7,3.3) {$\ell_1$};
\node(ell2) at (15.4,2.2) {$\ell_2$};

\node(b3) at (12, -2) {(b)};

\end{tikzpicture}
\end{center}

\caption{ (a) A subdivision of  $\hat{T}_{2,2}$ and  (b) a subdivision of  $\hat{T}_{2,3}$  that is constructed from two copies of subdivisions of  $\hat{T}_{2,2}$.
Each wavy gray edge indicates a subdivided edge.} 

\label{fig-depth-3-tree}
\end{figure}

We begin with the case of binary trees where the construction is easier to visualize.

\begin{thm} 
There exists a subdivision of $\hat{T}_{2,d}$ with treatment number at most 2.
\label{T2d-thm}
\end{thm} 

\begin{proof}
It is straightforward to find a width 2 protocol that clears $\hat{T}_{2,1}$ and for which the stem is treated only in the last time-step.  If we let $H_1 = H_2 = \hat{T}_{2,1}$ and apply the construction in the proof of Theorem~\ref{binary}, we obtain a subdivision of  $\hat{T}_{2,2}$ and a width 2 protocol that clears it for which the stem   is treated only in the last time-step.   Similarly, if we let $H_1 $ and $H_2$ each equal  this new subdivision of  $\hat{T}_{2,2}$ and apply the construction in the proof of Theorem~\ref{binary}, we obtain a subdivision of  $\hat{T}_{2,3}$ and a width 2 protocol that clears it for which the stem   is treated only in the last time-step.  This is illustrated in Figure~\ref{fig-depth-3-tree}.  Continuing by induction we obtain a subdivision of  $\hat{T}_{2,d}$ and a width 2 protocol that clears it for which the stem   is treated only in the last time-step.  
\end{proof}

\begin{cor}
Every binary tree has a subdivision with treatment number at most 2.
\label{binary-tree-cor}
\end{cor}

\begin{proof}
Let $T$ be a binary tree, so $T$ is a subgraph of $\hat{T}_{2,d}$ for some $d$.  By Theorem~\ref{T2d-thm} there exists a graph $W$ that is a subdivision of 
$\hat{T}_{2,d}$  and for which $\tau(W) \le 2$.  We subdivide the corresponding edges of $T$ to obtain a subdivision $T'$ of $T$ that is a subgraph of $W$.  By Theorem~\ref{subgraph-thm} we know $\tau(T') \le i(W)$ and hence $\tau(T') \le 2$.  Thus we have a subdivision of $T$ with treatment number at most 2 as desired.
\end{proof}

We next extend our result to include $m$-ary trees. Table~\ref{j-table} shows the order in which subgraphs $H_k$ and paths $P_k$ are cleared in the proof of Theorem~\ref{m-subt}. Observe that  $\mathcal{J}_4$ is obtained from $\mathcal{J}_3$ by starting with the sequence $H_4\; P_4$ and then inserting $P_4$ after each occurrence of $P_3$. 

\begin{table}[htbp]
\begin{center}
\begin{tabular}{|c|ccccccccccccc|} \hline 
$\mathcal{J}_3$ &  & & $H_3$ & $P_3$ & & $H_2$ & $P_2$ & $P_3$ & & $H_1$ & $P_2$ & $ P_3$ & \\ \hline
$\mathcal{J}_4$ & $H_4$ & $\mathbf P_4$ & $H_3$ & $P_3$ &$\mathbf P_4$ & $H_2$ & $P_2$ & $P_3$ & $\mathbf P_4$& $H_1$ & $P_2$ & $ P_3$ & $\mathbf P_4$\\ \hline

\end{tabular}
\end{center} 
\caption{The order in which portions of graph $H$ are cleared in protocols $\mathcal{J}_3$ and $\mathcal{J}_4$ in the proof of Theorem~\ref{m-subt}.}
\label{j-table} 
\end{table}

\begin{thm} \label{m-subt}
For $1\leq k\leq m$, let $H_k$ be a graph with a leaf $\ell_k$. 
Let $J_k$ be a protocol of width 2 that clears $H_k$ and takes $s_k$ time-steps, and for which   $\ell_k $ is treated only in the last time-step.  Let $s = \max \{s_1, s_2, \ldots, s_m\}$.   Let $H$ be $H_1+H_2+\cdots +H_m$ with a new leaf $\ell$ adjacent to $\ell_1$,  and, for each $k$, $2\leq k\leq m$,  a new path $P_k$ between $\ell_1$ and $\ell_k$ with 
$2^{k-2}s$ new interior vertices. Then there is a protocol of width 2 that clears $H$ and for which $\ell$ is treated only in the last time-step.
\end{thm}

\begin{proof}
We proceed by induction on $m$. By the proof of Theorem~\ref{binary}, we have a protocol for $m=2$ and $s = \max\{s_1,s_2\}$ in which  $\ell$ is treated only in the last time-step, and using $J_1$ and $J_2$, each of $H_1$ and $H_2$ is cleared 
exactly once,  $P_2$ is cleared before $H_1$ and also after each time $P_1$ is cleared.  This is the base case.  By induction, assume that $m\geq 3$ and $\mathcal{J}(m-1)$ is a protocol for $m-1$ subgraphs, where $\ell$ is treated only in the last time-step, each subgraph $H_i$ is cleared 
exactly once, in order from $H_{m-1}$ to $H_1$, and the paths $P_i$  are cleared as in the proof of Theorem~\ref{binary}. Our plan is to insert subsequences into $\mathcal{J}(m-1)$ to create a protocol for $H$. First, we start by following $J_m$ in time at 
most $s$. Then we clear the interior vertices of $P_m$ and $\ell_1$, from $\ell_m$ towards $\ell_1$, in time at most $s\cdot 2^{m-2}$.  

Then we do $\mathcal{J}(m-1)$ with these additions: each time that we finish clearing vertices on $P_{m-1}$, we insert a sequence that clears the non-green interior vertices of $P_m$, and treats $\ell_1$ every other time to keep $\ell_1$ green.

By induction, the vertices of $H_{m-1}$ do not become yellow again after we do the $J_{m-1}$ protocol. However, we do have to clear the interior vertices repeatedly, because the color change to yellow and then red 
spreads from $\ell_1$. Each time we clear $P_{m-1}$ in $\mathcal{J}(m-1)$, the color change 
has not reached $\ell_{m-1}$ which is in $H_{m-1}$. Hence, the time since $P_{m-1}$ was last cleared is less than or equal to $s\cdot 2^{m-3}$, which is the number of interior vertices of $P_{m-1}$.  By construction, the interior vertices of $P_{m-1}$ and $P_m$ start to change color at the same time. The source of the color change
is $\ell_1$. Thus, the number of non-green interior vertices of $P_m$ is the number of non-green vertices on $P_{m-1}$ plus the number of time-steps to clear all interior vertices of $P_{m-1}$. The total is at most $2\cdot s\cdot 2^{m-3}=s\cdot 2^{m-2}$. Hence $\ell_m$ does not change color before the clearing of $P_m$ is completed, and $H_m$ does not need to be cleared again.  This completes the proof that the new protocol clears $H$, including treating the vertex $\ell$  only in the last round.
\end{proof}

\begin{thm} \label{time-bound}
The protocol constructed in the proof of Theorem~\ref{m-subt}
 takes at most $s((m+3)2^{m-2}-1)$ time-steps.
\end{thm}

\begin{proof}
We compute the number of time-steps to clear $H$ in the proof of Theorem~\ref{m-subt}, that is, the number of time-steps in $\mathcal{J}(m)$. 
Table~\ref{j-table}  illustrates the relationship between $\mathcal{J}(3)$ and $\mathcal{J}(4)$.  In $\mathcal{J}(m)$,
it takes $4s$ time-steps to clear $H_2, P_2, H_1$ and $P_2$ a second time. Recursively, for $3\leq k\leq m$, $P_k$ is cleared one more time than $P_{k-1}$. Since $P_2$ is cleared twice, by induction $P_k$ is cleared $k$ times for $2\leq k\leq m$. Each time we clear $P_k$, we need only treat its compromised vertices. The number of time-steps of $\mathcal{J}(k)$ equals the number of time-steps of $
\mathcal{J}(k-1)$ plus the additional  time to clear 
$H_k$ and the compromised vertices of $P_k$ in $k$ different clearings. For $3\leq  k\leq m$, since the length of $P_k$ is $s\cdot 2^{k-2}$,  it takes $s\cdot 2^{k-2}$ time-steps to clear $P_k$ the first time in the protocol. The second time that $P_k$ is cleared, the number of compromised vertices on $P_k$ is  the number of time-steps it takes to clear $H_{k-1}$ and $P_{k-1}$, which is $s + s\cdot 2^{k-3}$. For the third and subsequent clearings, the 
number of time-steps it takes to clear $P_k$ is twice the number of time-steps it takes to clear $P_{k-1}$ immediately before $P_k$. Note that by Theorem~\ref{binary}, the protocol $\mathcal{J}(2)$ is completed in $4s$ time-steps.
Define $f(k)$ as follows: let $f(0)=0, f(1)=s$ and $f(2)=3s$, and for $3\leq k\leq m$, let $f(k)$ be the number of time-steps of 
$\mathcal{J}(k)$ minus the number of time-steps of $\mathcal{J}(k-1)$. Thus $f(k)$ is the number of time-steps in protocol  $\mathcal{J}(k)$  to clear $H_k$ once, and $P_k$  $k$ times.  
The initial values are chosen so that $f(0)+f(1)+f(2)$ is the number of time-steps in protocol $\mathcal{J}(2)$.  By telescoping, the total time of $\mathcal{J}(m)$ is $ \sum_{i=0}^m f(m)$. 
For $3\leq k\leq m$, define $g(k)$ to be the number  of time-steps to clear $P_k$ in all but the first clearing.  Thus $f(k) - g(k)$ is the number of time-steps to clear $H_k$ and the first clearing of $P_k$ in $\mathcal{J}(k)$, so we get   $f(k) - g(k) = s+s\cdot 2^{k-2}$.  Substituting $k-1$ for $k$ and rearranging, we get $g(k-1) = f(k-1) - s - s\cdot 2^{k-3}$.   By the proof of Theorem~\ref{m-subt}, the number of time-steps to clear $P_k$ for the $i$th time, where $3 \le i \le k$, is twice as large as the number of time-steps to clear $P_{k-1}$ for the $(i-1)$st time.
Therefore,  $f(k) - 2g(k-1)$  is the number of time steps in  $\mathcal{J}(k)$ to clear $H_k$ (once), and $P_k$ the first and second times. Thus $f(k) - 2g(k-1) = s+s\cdot 2^{k-2}+(s+s\cdot 2^{k-3})$.   Now substituting in for $g(k-1)$ from above we get  
 $f(k)=s\cdot 2^{k-3} + 2 f(k-1)$, for $k\geq 3$.

We now use generating functions to complete the computation of the number of time-steps of $\mathcal{J}(m)$. Let  $F(x) = \sum_{i=0}^\infty f(i)x^i$. Using the recurrence, we have $F(x) -s x - 3s x^2= \frac{sx^3}{(1-2x)}+ 2xF(x) -2sx^2 $, so 
$F(x)= \frac{sx^3}{(1-2x)^2}+ \frac{sx+sx^2}{1-2x}=\frac{sx(1-x-x^2)}{(1-2x)^2}$. 
The total time of $\mathcal{J}(m)$ is $\sum_{i=0}^m f(m)$. Its generating function is $F(x)/(1-x) =\frac{sx(1-x-x^2)}{(1-x)(1-2x)^2}$.  Using partial fractions, the $m$th coefficient of $\frac{F(x)}{1-x}$ is $c_1 + c_2 2^m+ c_3 (m+1) 2^m$ for some constants $c_1, c_2, c_3$. Using the initial values, the total time of $\mathcal{J}(m)$ is  $s((m+3)2^{m-2}-1).$ 
\end{proof}

We use the construction in Theorem~\ref{m-subt} to prove a a generalization of Theorem~\ref{T2d-thm}.

\begin{thm} 
There exists a subdivision of $\hat{T}_{m,d}$ with treatment number at most 2.
\label{Tmd-thm}
\end{thm} 

\begin{proof}
The proof is analogous to that of Theorem~\ref{T2d-thm}, using Theorem~\ref{m-subt} in place of Theorem~\ref{binary}.
\end{proof}

We now present our main result of this section.

\begin{cor}\label{cor:tree} For any tree $T$, there is a subdivision of $T$  with treatment number at most 2.
\end{cor}
\begin{proof} 
Let $T$ be a tree, so $T$ is a subgraph of $\hat{T}_{m,d}$ for some  $m$ and $d$.  By Theorem~\ref{Tmd-thm} there exists a graph $W$ that is a subdivision of 
$\hat{T}_{m,d}$  and for which $\tau(W) \le 2$.  Analogous to the proof of Corollary~\ref{binary-tree-cor},
there exists a subdivision $T'$ of $T$ that is also a subgraph of $W$, and hence  $\tau(T') \le i(W) \le 2$ as desired.
\end{proof}

To conclude, we note the following connections to the inspection number of a graph, defined in Section~\ref{sec:intro}. The inspection number equals the treatment number when $r=1$ and $s=0$.  Bernshteyn and Lee~\cite{BL21} have studied subdivisions for the inspection number.  They showed that the inspection number can both increase or decrease when a graph is changed by taking a subdivision, and prioritized the characterization of graphs for which there exists a subdivision that results in a graph with inspection number at most three.

\section{Questions}

\label{sec:ques} 

We conclude with some open questions.

\begin{ques} 
In Theorem~\ref{time-bound}, we gave a bound for the number of time-steps needed to clear the graph $H_k$. 
What is the smallest number of time-steps needed to clear $H_k$? 
\end{ques}

For the case where $r=s=1$, Theorem~\ref{one-caterpillar} shows  that  there is a protocol of width 1 that clears graph $H$ (i.e., $\tau(H)=1$) if and only if $H$ is a caterpillar.  We ask more generally about the existence of width 1 protocols.

\begin{ques} For $r \geq 1, s \geq 1$, can we characterize those graphs $H$ for which a protocol of width 1 clears $H$? For general $r$ and $s$, can we characterize graphs for which $\tau_{r,s}(H)=1$?
\end{ques} 

More generally, for a graph $H$, we ask about the distinction between $\tau_{r,s}(H)$ and $\tau_{s,r}(H)$ when $r$ and $s$ are not both 1.  

\begin{ques} Let $H$ be a graph.  Is it true that $\tau_{r,s}(H) = \tau_{s,r}(H)$ for all $r$ and $s$?
\end{ques}

In Section~\ref{sec:subdivisions},  Corollary~\ref{cor:tree} shows that for a tree $T$ and $r=s=1$, there exists a subdivision of $T$ that results in a graph with treatment number at most 2.  Suppose that $T$ is a tree, but not a caterpillar.  By Theorem~\ref{one-caterpillar}, we know $\tau(T) \geq 2$. 

\begin{ques} What is the smallest tree $T$ for which $\tau(T)=3$? (Here, smallest could refer to either the minimum cardinality of a vertex set or the diameter.)  \end{ques}

More generally, we ask about subdivisions of general graphs.

\begin{ques} Let $H$ be a graph and fix $k \in \mathbb{N}$. What are the smallest values of $r$ and $s$ for which there is a subdivision $H'$ of $H$  where $\tau_{r,s}(H') \leq k$?\end{ques}

The treatment number $\tau(H)$ has some connections to a cops and robbers game called the \emph{one-proximity game.}  Bonato et.\ al.\ \cite{one-vis} study the \emph{one-proximity number} of a graph $H$, denoted by $\prox_1(H)$, which is the number of cops needed to guarantee a win for the cops when the game is played on a graph $H$.  It is straightforward to show that any treatment protocol $(A_1, A_2, \ldots, A_N)$ that clears a graph $H$ immediately  leads to a solution to the one-proximity game on $H$. This can be accomplished by positioning the cops on the vertices in $A_j$ in round $j$ for  $1 \le j \le N$.    It follows that $ 1 \le \prox_1(H) \le \tau(H)$ for all graphs $H$.  However, these quantities can be arbitrarily far apart, for example, the complete graph $K_n$ has $\prox_1(K_n) =1$ while $\tau(K_n) = \lceil \frac{n}{2} \rceil$.

\begin{ques} If $\tau(H) = 1$ then $\prox_1(H) = 1$ and as shown in Theorem~\ref{one-caterpillar}, this class of graphs is exactly the caterpillars. For which additional graphs $H$ does 
$\tau(H) = \prox_1(H)$?

\end{ques}

\section{Acknowledgements}

Our work on the discrete-time treatment model was inspired by an open problem regarding the inspection game, presented by A. Bernshteyn at the Graph Theory Workshop at Dawson College (Qu\'{e}bec) in June 2023.
 N.E.~Clarke acknowledges research support from NSERC (2020-06528). M.E. Messinger acknowledges research support from NSERC (grant application 2018-04059).  The authors thank Douglas B. West for his valuable insights during the early stages of this project.

\end{document}